\documentclass[11pt]{amsproc} 

\usepackage{amsmath, amssymb, amsthm, setspace, mathtools, graphicx, enumitem, nicefrac}
\usepackage{thmtools}
\usepackage[usenames]{color}
\usepackage{tikz-cd}
\usepackage{microtype}

\usepackage[pagebackref,colorlinks,linkcolor=blue,citecolor=blue,urlcolor=blue,hypertexnames=true]{hyperref}
\usepackage{amsrefs} 

\declaretheorem[style=plain,name=Theorem]{theorem}
\declaretheorem[style=plain,numberlike=theorem,name=Lemma]{lemma}

\declaretheorem[style=remark,numberlike=theorem,name=Remark]{remark}

\declaretheorem[style=remark,numberlike=theorem,name=Question]{question}

\newcommand{\N}{\mathbb{N}}             
\newcommand{\Z}{\mathbb{Z}}             
\newcommand{\R}{\mathbb{R}}             
\newcommand{\SL}{\mathrm{SL}}           
\newcommand{\GL}{\mathrm{GL}}           
\newcommand{\Aut}{\mathrm{Aut}}         
\newcommand{\SAut}{\mathrm{SAut}}       
\newcommand{\one}{\mathbf{1}}           
\newcommand{\I}{\mathrm{I}}             
\newcommand{\inv}[1]{\bar{#1}}          
\newcommand{\E}[1]{{E_{\smash{#1}}}}    

\DeclareMathOperator{\Ima}{Im}
\DeclareMathOperator{\id}{id}
\DeclareMathOperator{\supp}{supp}
\DeclareMathOperator{\colim}{colim}

\author{Martin Nitsche}
\address{Martin Nitsche, Karlsruhe Institute of Technology, Germany}
\email{martin.nitsche@kit.edu}

\title{A human property~(T) proof for high-rank $\Aut(F_n)$}

\allowdisplaybreaks
\onehalfspace

\begin{document}

\begin{abstract}
Existing property~(T) proofs for $\Aut(F_n)$, $n\geq 4$, rely crucially on extensive computer calculations. We give a new proof that $\Aut(F_n)$ has property~(T) for all but finitely many $n$ that is inspired by the semidefinite programming approach but does not use the computer in any step. More specifically, we prove property~(T) for a certain extension $\Gamma_n$ of $\SAut(F_n)$ as $n\to\infty$.
\end{abstract}

\maketitle

\section{Introduction}

Let $\Gamma$ be a discrete group with finite generating set $\mathcal{S}=\mathcal{S}^{-1}$.
Then $\Gamma$ has Kazhdan's property~(T) if any unitary representation with almost invariant vectors contains a (nontrivial) invariant vector, or, equivalently, if the reduced first cohomology $\bar{\mathrm{H}}^1(\Gamma,\pi)$ vanishes for every unitary representation $\pi$.
Property~(T) is a much studied rigidity property with many interesting consequences, but proving it for a given group is typically difficult.

In 2016, Ozawa \cite{O} proposed a new approach to prove property~(T) with the computer, and Netzer and Thom \cite{NT} showed how to implement the calculation in the framework of a semidefinite program (SDP).
The semidefinite programming approach proved to be successful in practice, in particular for the automorphism groups of the free groups, $\Aut(F_n)$ with $n\geq 4$ \cites{KNO,KKN,N}, that had long been suspected to have property~(T). These groups can be thought of as a generalization of the general linear groups $\GL(n,\Z)=\Aut(\Z^n)$, which classically have property~(T) for $n\geq 3$. The practical significance is that $\Aut(F_n)$ having property~(T) explains the fast mixing properties of the product replacement algorithm used in computational group theory \cite{LP}.

The property~(T) computer proof consists of a numerical computer calculation that produces a decomposition of $\Delta^2-\varepsilon\Delta$ as a sum $\sum_i {\xi_i}^*\xi_i$, $\xi_i\in\R\Gamma$, where $\varepsilon>0$ and $\Delta\vcentcolon=\sum_{S\in \mathcal{S}}\one-S\in\R\Gamma$ is the group Laplacian.
This process is quite intransparent. The decomposition provided by the computer is usually huge and seemingly without a pattern. But even if it were simple enough to check it manually -- say, two summands supported on ten group elements -- it is unclear if one could learn more from it than the fact that it works.

It is therefore our goal to find a property~(T) proof for $\Aut(F_n)$ that does not use the computer. This task appears to be very hard for a fixed small $n$, but if one lets $n\to\infty$, one obtains more freedom, which simplifies the task. We note that Ozawa \cite{O2} has similarly studied a property~(T)--like property for certain matrix groups $\mathrm{EL}_n(\mathcal{R})\subset\SL(n,\mathcal{R})$ over nonunital rings as $n\to\infty$. But his tools, in particular the representation theory of the integral Heisenberg group, do not transfer to $\Aut(F_n)$.

The benefit of our human proof is that one can see which properties of $\Aut(F_n)$ are used in which steps. And one can try to generalize all or part of the proof to other groups that are structurally similar to $\Aut(F_n)$. In fact, our proof does not use all the relations satisfied by $\Aut(F_n)$, whence we actually prove property~(T) for a certain group extension $\Gamma_n$ of a finite-index subgroup of $\Aut(F_n)$, defined in section~\ref{sec:setting}.

\begin{theorem}\label{thm:main-result}
All but finitely many of the groups $\Gamma_n$ have property~(T). Consequently, the same is true for the groups $\Aut(F_n)$.
\end{theorem}

Our proof was inspired by the geometric dual \cite{N} of the property~(T) SDP, which we use to disprove the existence of harmonic cocycles for $\Gamma_n$, $n\gg 3$. The strategy is to find more and more constraints that a harmonic cocycle must satisfy, until we arrive at a contradiction. To obtain these constraints, we employ a range of arguments where both the number of group elements in the support of a group algebra element $\xi\in\R\Gamma$ and the word metric of the group elements can go to infinity.
While these arguments were found with a human proof, the resulting constraints can also be added into the SDP of the computer proof in order to make it much simpler. A second benefit of our proof is therefore that we showcase a range of strategies to simplify the asymptotic property~(T) proof for any sequence of groups with a similar symmetry as $\Aut(F_n)$.

Finally, we remark that, in theory, our proof is very much constructive. The equality constraints that we obtain for $n\to\infty$ are the limit of inequalities that can all be quantified. By doing this with sufficient care, one could find an explicit $N\in\N$ such that $\Gamma_n$ is shown to have property~(T) for all $n\geq N$, and one would even find an (inefficient) lower bound for the Kazhdan constants. We have not attempted this, as it would defeat the purpose of making the proof as transparent as possible.
Furthermore, all steps in our proof could be dualized, as explained in \cite{N}, and stated in the language of order units for the augmentation ideal of $\R\Gamma$, as in \cite{O2}.

\section{Preliminaries on harmonic cocycles}
We quickly recapitulate how proving property~(T) can be reformulated as disproving the existence of harmonic cocycles. More details can be found in \cite{N} and \cite{BHV}*{Sections 2.2, 2.10 and 3.2}.

When $\Gamma$ is a discrete group and $\pi\colon\Gamma\to\mathcal{H}$ an orthogonal representation on a real Hilbert space, then a 1-cocycle for $(\Gamma,\pi)$ is a map $c\colon\Gamma\to\mathcal{H}$ that satisfies the relation $c(\gamma_1\gamma_2)=c(\gamma_1)+\pi(\gamma_1)c(\gamma_2)$. This also implies $c(\one)=0$.
Such a cocycle gives rise to a positive semidefinite form $\langle\cdot,\cdot\rangle_c$ on $\I\Gamma\vcentcolon=\big\{\sum_i\lambda_i\gamma_i\,\big\vert\,\gamma_i\in\Gamma,\lambda_i\in\R,\sum_i\lambda_i=0\big\}$, the augmentation ideal of $\R\Gamma$, by $\langle \sum_i\lambda_i\gamma_i,\sum_k\mu_k\gamma_k\rangle_c\vcentcolon=\langle\sum_i\lambda_i c(\gamma_i),\sum_k\mu_k c(\gamma_k)\rangle_\mathcal{H}$. This form is nontrivial if $c$ is nontrivial (i.e., not identically zero), and by the cocycle condition it is invariant under left multiplication with elements of $\Gamma$. Up to orthogonal equivalence of representations and enlarging of the Hilbert space, the representation $\pi$ and the cocycle $c$ can be reconstructed from $\langle\cdot,\cdot\rangle_c$, but we will not need this. The information of $\langle\cdot,\cdot\rangle_c$ can also be encoded in the function $\Phi\colon\Gamma\to\R_{\geq 0},\Phi(\gamma)\vcentcolon=\|\one-\gamma\|^2$. A function that arises like this from an invariant positive semidefinite form on $\I\Gamma$ is called \emph{conditionally of positive type}. Note that always $\Phi(\one)=0$ and $\Phi(\gamma^{-1})=\Phi(\gamma)$.

When $\mathcal{S}$ is a generating set for $\Gamma$, every 1-cocycle is determined by its values on $\mathcal{S}$, whence the space $Z^1$ of all 1-cocycles can be viewed as a closed subspace of the Hilbert space~$\mathcal{H}^\mathcal{S}$.
Elements in the image of the coboundary map $\delta^0\vcentcolon=\bigoplus_{S\in \mathcal{S}}\id-\pi(S)\colon\mathcal{H}\to\mathcal{H}^{\mathcal{S}}$ are trivially cocycles. The quotient $Z^1/\overline{\Ima\delta^0}$ is denoted $\bar{\mathrm{H}}^1(\Gamma,\pi)$ and can be identified with the space of harmonic cocycles, $\bar{\mathrm{H}}^1(\Gamma,\pi)\cong Z^1\cap(\Ima\delta^0)^\perp=Z^1\cap\ker(\delta^0)^*$.

Assume now that $\Gamma$ has a finite and symmetric generating set $\mathcal{S}=\mathcal{S}^{-1}$.
Shalom's theorem states that if $\Gamma$ does \emph{not} have property~(T), then it must allow an orthogonal representation $\pi$ with a nontrivial harmonic cocycle $c$. From this one then obtains a nontrivial left-invariant positive semidefinite form $\langle\cdot,\cdot\rangle$ on $\I\Gamma$ that satisfies
\begin{align}\label{eqn:harmonic}
\begin{split}
0=\|(\delta^0)^*(c)\|^2_{\mathcal{H}^\mathcal{S}}
&=\big\|{\textstyle\sum}_{S\in \mathcal{S}}\,c(S)-\pi(S)^*c(S)\big\|^2_{\mathcal{H}^\mathcal{S}}\\
&=\big\|{\textstyle\sum}_{S\in \mathcal{S}}\,c(S)+c(S^{-1})-\big(c(S^{-1})+\pi(S^{-1})c(S)\big)\big\|^2_{\mathcal{H}^\mathcal{S}}\\
&=\big\|{\textstyle\sum}_{S\in \mathcal{S}}\,c(S)+c(S^{-1})-c(S^{-1}S)\big\|^2_{\mathcal{H}^\mathcal{S}}\\
&=\big\|2\cdot{\textstyle\sum}_{S\in \mathcal{S}}\,c(S)-c(\one)\big\|^2_{\mathcal{H}^\mathcal{S}}\\
&=4\cdot\big\|{\textstyle\sum}_{S\in \mathcal{S}}\,\one- S\big\|^2.
\end{split}
\end{align}
This is the statement that we want to lead to a contradiction.

\section{The setting for \texorpdfstring{$\SAut(F_n)$, $n\to\infty$}{SAut(Fn), n->inf}}\label{sec:setting}
Let $F_n=\langle \mathcal{Z}_n\rangle$ be the free group with generating set $\mathcal{Z}_n=\{z_1,\dots,z_n\}$. We use the notation $\inv{z}$ to denote the inverse of an element $z\in\mathcal{Z}_n\cup\mathcal{Z}_n^{-1}$ and we use $\tilde{z}\vcentcolon=\{z,\inv{z}\}$ if we do not want to distinguish between a generator and its inverse.
Let $\Aut(F_n)$ be the automorphism group of $F_n$, $q\colon\Aut(F_n)\to\GL(n,\Z)$ the group homomorphism induced by the abelianization map $F_n\to\Z^n$ and $\SAut(F_n)=q^{-1}(\SL(n,\Z))$. Since $\SAut(F_n)$ is a finite index subgroup of $\Aut(F_n)$, it has property~(T) exactly if $\Aut(F_n)$ has (T).

Let $\mathcal{S}_n\vcentcolon=\big\{\E{ab}\mid a,b\in \mathcal{Z}_n\cup\mathcal{Z}_n^{-1}, \tilde{a}\neq\tilde{b}\big\}$, where $\E{ab}\colon F_n\to F_n$ denotes the automorphism that sends $a$ to $ab$ and that is the identity on $\mathcal{Z}_n\cup\mathcal{Z}_n^{-1}\setminus\{a,\inv{a}\}$.
The set $\mathcal{S}$ is symmetric and generates $\SAut(F_n)$.
Gersten \cite{G} showed that the following relations -- for all $a,b,c,d\in \mathcal{Z}_n\cup\mathcal{Z}_n^{-1}$ with $\tilde{a},\tilde{b},\tilde{c},\tilde{d}$ distinct -- give a presentation for $\SAut(F_n)$, $n\geq 3$.
\begin{align}\stepcounter{equation}
\E{ab}\E{a\inv{b}}&=\one\label{eqn:inverse}\tag{\arabic{equation}a}\\
[\E{ab},\E{cd}]&=\one\label{eqn:disjoint-commutes}\tag{\arabic{equation}b}\\
[\E{ab},\E{\inv{a}c}]=[\E{ab},\E{c\inv{b}}]&=\one\label{eqn:same-type-commutes}\tag{\arabic{equation}c}\\
[\E{ab},\E{bc}]&=\E{ac}\label{eqn:pentagram-relation}\tag{\arabic{equation}d}\\
[\E{ab},\E{cb}]&=\one\label{eqn:right-commutes}\tag{\arabic{equation}e}\\
[\E{ab},\E{\inv{a}b}]&=\one\notag\\
\E{ba}\E{\inv{a}b}\E{\inv{b}\inv{a}}\E{b\inv{a}}\E{ab}\E{\inv{b}a}&=\one\notag\\
(\E{ba}\E{\inv{a}b}\E{\inv{b}\inv{a}})^4&=\one\notag
\end{align}
Here, $[X,Y]\vcentcolon=XYX^{-1}Y^{-1}$.
The last three relations are not needed to prove property~(T).
So let $\Gamma_n$ be the group with generators $\mathcal{S}_n$ and relations~\ref{eqn:inverse} to~\ref{eqn:right-commutes}.

We will think of each generator $\E{ab}$ as a connection between the \emph{places} $\tilde{a},\tilde{b}$ and we will say that the \emph{support} of a word over $\mathcal{S}_n$ consists of all the places that appear in its letters. It follows from equation~\ref{eqn:disjoint-commutes} that words with disjoint support commute in $\Gamma_n$.

When $\sigma\colon \mathcal{Z}\to \mathcal{Z}$ is a permutation and $s\colon\mathcal{Z}\to\{\pm 1\}$ a choice of signs, the group homomorphism $\Psi_{\sigma s}\colon\Gamma_n\to\Gamma_n$ defined by $\E{ab}\mapsto \E{\sigma(a)^{s(a)}\sigma(b)^{s(b)}}$ is an automorphism of $\Gamma_n$. The group $\Lambda_n<\Aut(\Gamma_n)$ of all such automorphisms is finite and fixes $\mathcal{S}_n$ set-wise.
Any nontrivial left-invariant positive semidefinite linear form $\langle\cdot,\cdot\rangle_n$ on $\I\Gamma_n$ can be averaged over $\Lambda_n$ in order to make it invariant under $\Lambda_n$. We will always assume that this has been done and that the form has been normalized by scaling such that $\|S-\one\|_n=\|\E{z_1z_2}-\one\|_n=1$ for all $S\in\mathcal{S}_n$.
If $\omega=S_1\dots S_k$ is a word in $\mathcal{S}$ of length $|\omega|$, then the triangle inequality provides a bound
\begin{equation}\label{eqn:quadratic-bound}
\|\omega-\one\|_n\leq\sum_i \|S_1\dots S_{i+1}-S_1\dots S_i\|_n=\sum_i\|S_{i+1}-\one\|_n=|\omega|
\end{equation}
that does not depend on the specific cocycle. We can use this bound to pass to the limit $n\to\infty$. For $a,b\in \mathcal{Z}$ let $\Delta_{ab}\vcentcolon=8\cdot\one-\E{ab}-\E{a\inv{b}}-\E{\inv{a}b}-\E{\inv{a}\inv{b}}-\E{ba}-\E{b\inv{a}}-\E{\inv{b}a}-\E{\inv{b}\inv{a}}$.

\begin{lemma}\label{lem:limit-form}
Assume that infinitely many of the groups $\Gamma_n$ do not have property~(T).
Then there exists a nontrivial left-invariant positive semidefinite form on the augmentation ideal of the group algebra of $\Gamma=\colim\Gamma_n$ that is invariant under $\Lambda=\colim \Lambda_n$ and satisfies
\begin{align}
\label{eqn:disjoint-exact}
\langle\Delta_{ab},\Delta_{cd}\rangle=\langle\one-\E{ab},\one-\E{cd}\rangle=\|\E{ab}\E{cd}-\E{ab}-\E{cd}+\one\|^2&=0,\\
\label{eqn:adjacent-big}
\langle\Delta_{ab},\Delta_{ac}\rangle&\leq 0
\end{align}
for all $\tilde{a},\tilde{b},\tilde{c},\tilde{d}$ distinct.
\end{lemma}
\begin{proof}
Using the assumption, let $\mathcal{N}=\{n\in\N\mid\Gamma_n\text{ does not have property~(T)}\}$, and for each $n\in\mathcal{N}$ let $\langle\cdot,\cdot\rangle_n$ be the bilinear form associated to a normalized $\Lambda_n$-invariant harmonic cocycle on $(\Gamma_n,\mathcal{S}_n)$. For every choice $\xi_1,\xi_2\in\I\Gamma$ the inner product $\langle\xi_1,\xi_2\rangle_n$ is defined for all but finitely many $n\in\mathcal{N}$, and by equation~\ref{eqn:quadratic-bound} these values lie in a compact subset of $\R$. Therefore, a subsequence of $(\langle\cdot,\cdot\rangle_n)_{n\in\mathcal{N}}$ converges in weak topology to a bilinear form $\langle\cdot,\cdot\rangle$ on $\I\Gamma$. The limit form is still positive semidefinite, left-invariant, $\Lambda$\nobreakdash-invariant and normalized by $\|\one-\E{ab}\|^2=1$.

To obtain equations~\ref{eqn:disjoint-exact} and~\ref{eqn:adjacent-big}, observe that by equation~\ref{eqn:harmonic} and $\Lambda$-invariance,
\begin{align}\label{eqn:harmonicity}
\begin{split}
0\overset{\ref{eqn:harmonic}}=&\big\|2\cdot{\textstyle\sum}_{\mathcal{S}_n}\one-S\big\|_n^2
=\langle{\textstyle\sum}_{a\neq b\in\mathcal{Z}_n}\Delta_{ab},{\textstyle\sum}_{c\neq d\in\mathcal{Z}_n}\Delta_{cd}\rangle
=\sum_{a\neq b}\sum_{c\neq d}\langle\Delta_{ab},\Delta_{cd}\rangle_n\\
=&\sum_{a,b,c,d\text{ distinct}}\langle\Delta_{ab},\Delta_{cd}\rangle_n+4\cdot\sum_{a,b,c\text{ distinct}}\langle\Delta_{ab},\Delta_{ac}\rangle_n+2\cdot\sum_{a\neq b}\|\Delta_{ab}\|_n^2\\
=&\phantom{{}+4}n(n-1)(n-2)(n-3)\cdot\langle\Delta_{z_1z_2},\Delta_{z_3z_4}\rangle_n\\
&+4n(n-1)(n-2)\cdot\langle\Delta_{z_1z_2},\Delta_{z_1z_3}\rangle_n\\
&+2n(n-1)\cdot\|\Delta_{z_1z_2}\|_n^2.
\end{split}
\end{align}
As $n\to\infty$, the dominating term on the right-hand side has to go to $0$.
For $a,b\in\mathcal{Z}$ let $\mathcal{E}_{ab}\subset\mathcal{S}$ be the set of all generators that appear in $\Delta_{ab}$. If $\tilde{a},\tilde{b},\tilde{c},\tilde{d}$ are distinct, all products $\{XY\mid X\in\mathcal{E}_{ab},Y\in\mathcal{E}_{cd}\}$ are related to $\E{ab}\E{cd}$ by $\Lambda$-symmetry. Hence, the dominating term can be simplified to
\[
\langle\Delta_{z_1z_2},\Delta_{z_3z_4}\rangle_n=\sum_{X\in\mathcal{E}_{z_1z_2}}\sum_{Y\in\mathcal{E}_{z_3z_4}}\langle\one-X,\one-Y\rangle_n=64\cdot\langle\one-\E{z_1z_2},\one-\E{z_3z_4}\rangle_n,
\]
and since $X\vcentcolon=\E{z_1z_2}$ commutes with $Y\vcentcolon=\E{z_3z_4}$, this is a positive multiple of
\[\arraycolsep=0.5pt\begin{array}{rllll}
0\leq&\|XY-X-Y+\one\|_n^2&&&\\
=&\langle XY-X,XY-Y\rangle_n
&-\langle XY-X,X-\one\rangle_n
&-\langle Y-\one,XY-Y\rangle_n
&+\langle Y-\one,X-\one\rangle_n\\
=&\langle \one-Y^{-1},\one-X^{-1}\rangle_n
&-\langle Y-\one,\one-X^{-1}\rangle_n
&-\langle \one-Y^{-1},X-\one\rangle_n
&+\langle Y-\one,X-\one\rangle_n\\
=&4\cdot\langle \one-Y,\one-X\rangle_n.&&&
\end{array}\]
This also means that the dominating term on the right-hand side of equation~\ref{eqn:harmonicity} is always nonnegative, whence the second-most dominating term, $\langle\Delta_{z_1z_2},\Delta_{z_1z_3}\rangle_n$, must converge to a nonpositive value.
\end{proof}

In the next section we will show that a bilinear form as in Lemma~\ref{lem:limit-form} cannot exist.

\begin{remark}
Although $\Gamma_n$ does not have harmonic cocycles except for finitely many $n$, there do exist cocycles with $\langle\Delta_{ab},\Delta_{cd}\rangle=0$.
These can be constructed from the group representation \[\Gamma_n\to\SAut(F_n)\to\SL(n,\Z)\to\SL(n,\Z/2\Z)\to\mathrm{Permutations}\big((\Z/2\Z)^n\big)\curvearrowright\R^{|(\Z/2\Z)^n|}\]
together with a suitably chosen cyclic vector.
Therefore, we must also make use of equation~\ref{eqn:adjacent-big} to disprove the assumption of Lemma~\ref{lem:limit-form}.
\end{remark}

\section{Proof of the main result}

Assume that infinitely many of the groups $\Gamma_n$ do not have property~(T) and let $\langle\cdot,\cdot\rangle$ be the normalized bilinear form resulting from Lemma~\ref{lem:limit-form}.
From the properties of this form we want to derive a contradiction.
Recall that $\Phi\colon\Gamma\to\R$ is defined by $\Phi(\gamma)=\|\one-\gamma\|^2$.

We begin with an immediate consequence of equation~\ref{eqn:disjoint-exact}.

\begin{lemma}
If $\omega,\omega_1,\omega_2$ are words in $\mathcal{S}=\colim\mathcal{S}_n$ with $\supp(\omega_1)\cap\supp(\omega_2)=\emptyset$, then
\begin{align}\stepcounter{equation}
\label{eqn:disjoint-left}
\Phi(\omega)+\Phi(\omega_1\omega_2\omega)-\Phi(\omega_1\omega)-\Phi(\omega_2\omega)&=0\tag{\arabic{equation}a}\quad\text{and}\\
\label{eqn:disjoint-right}
\Phi(\omega)+\Phi(\omega\omega_1\omega_2)-\Phi(\omega\omega_1)-\Phi(\omega\omega_2)&=0.\tag{\arabic{equation}b}
\end{align}
\end{lemma}
\begin{proof}
First, consider the case where $\omega_1=X^{-1}$ and $\omega_2=Y^{-1}$ consist of a single letter.
Then, by equation~\ref{eqn:disjoint-exact} and the Cauchy--Schwarz inequality, we have
\begin{align*}
0\smash{\overset{\ref{eqn:disjoint-exact},\text{CS}}{=}}\hspace{-0.15cm}&-2\cdot\langle XY-X-Y+\one,\omega-\one\rangle-2\langle X-\one,Y^{-1}-\one\rangle\\
=&
-2\cdot\langle XY-\one,\omega-\one\rangle
+2\cdot\langle X-\one,\omega-\one\rangle
+2\cdot\langle Y-\one,\omega-\one\rangle
-2\cdot\langle X-\one,Y^{-1}-\one\rangle\\
=&
-\big(\|XY-\one\|^2+\|\omega-\one\|^2-\|XY-\omega\|^2\big)
+\big(\|X-\one\|^2+\|\omega-\one\|^2-\|X-\omega\|^2\big)\\
&+\big(\|Y-\one\|^2+\|\omega-\one\|^2-\|Y-\omega\|^2\big)
-\big(\|X-\one\|^2+\|Y^{-1}-\one\|^2-\|X-Y^{-1}\|^2\big)\\
=&
-\big(\Phi(XY)+\Phi(\omega)-\Phi(X^{-1}Y^{-1}\omega)\big)
+\big(\Phi(X)+\Phi(\omega)-\Phi(X^{-1}\omega)\big)\\
&+\big(\Phi(Y)+\Phi(\omega)-\Phi(Y^{-1}\omega)\big)
-\big(\Phi(X)+\Phi(Y)-\Phi(XY)\big)\\
=&\Phi(\omega)+\Phi(X^{-1}Y^{-1}\omega)
-\Phi(X^{-1}\omega)
-\Phi(Y^{-1}\omega).
\end{align*}
To obtain equation~\ref{eqn:disjoint-left}, let $X_i$ and $Y_i$ be the $i$-th letter in $\omega_1$, respectively $\omega_2$, let $\omega_1^{\leq i},\omega_2^{\leq i}$ be the subwords of the first $i$-many letters, and let $\omega_{ij}\vcentcolon=\omega_1^{\leq i-1}\omega_2^{\leq j-1}\omega$. Then
\begin{align*}
&\Phi(\omega)+\Phi(\omega_1\omega_2\omega)-\Phi(\omega_1\omega)-\Phi(\omega_2\omega)\\
=&\sum_{i=1}^{|\omega_1|}\sum_{j=1}^{|\omega_2|}\Phi\big(\omega_1^{\leq i-1}\omega_2^{\leq j-1}\omega\big)+\Phi\big(\omega_1^{\leq i}\omega_2^{\leq j}\omega\big)-\Phi\big(\omega_1^{\leq i}\omega_2^{\leq j-1}\omega\big)-\Phi\big(\omega_1^{\leq i-1}\omega_2^{\leq j}\omega\big)\\
=&\sum_{i=1}^{|\omega_1|}\sum_{j=1}^{|\omega_2|}\Phi(\omega_{ij})+\Phi(X_iY_j\omega_{ij})-\Phi(X_i\omega_{ij})-\Phi(Y_j\omega_{ij})=\sum_{i=1}^{|\omega_1|}\sum_{j=1}^{|\omega_2|}0=0.
\end{align*}
Equation~\ref{eqn:disjoint-right} follows by taking the inverse of all group elements, using $\Phi(\gamma^{-1})=\Phi(\gamma)$.
\end{proof}

Equations~\ref{eqn:disjoint-left} and~\ref{eqn:disjoint-right} are already very powerful. If one inserts these constraints into the semidefinite program that is used for the property~(T) computer proof, the calculation simplifies immensely. For our human proof, however, this is only the first step.
Next, we obtain additional information about $\Phi$ by evaluating it on certain very long words.

To every word $\omega$ over $\mathcal{S}$ one can associate a graph $\mathcal{G}$ with vertex set $\supp(\omega)$ and with one edge $(\tilde{a},\tilde{b})$ for every letter $\E{ab}$ in $\omega$. It is easiest to calculate with those words whose undirected graph is a tree. In the following we focus on \emph{star-shaped} words, where one fixed place $\tilde{x}$ appears in every letter and no other place appears in more than one letter.
We will say that two letters in a star-shaped word belong to the same \emph{type} if they are related by an element of $\Lambda$ that fixes $x$. The four possible types are $\E{xz},\E{\inv{x}z},\E{zx},\E{z\inv{x}}$.

Despite their simplicity, the star-shaped words witness some crucial phenomena of $\Gamma$, like the fact that the addition of only a single letter turns $\prod_i \E{xz_i}$ into the seemingly very distant $\E{yx}\cdot\prod_i\E{xz_i}=\big(\prod_i\E{yz_i}\big)\big(\prod_i\E{xz_i}\big)\E{yx}$, and similarly for $\prod_i\E{z_ix}$. This observation was the starting point for the following lemma.

\begin{lemma}\label{lem:cancellation}
Let the word $\omega=\omega_1 X_1 X_2 \omega_2$ be star-shaped around $\tilde{x}$, where $\omega_1,\omega_2$ are words in $\mathcal{S}$ and $X_1,X_2\in \mathcal{S}$ have the same type. Then
\begin{equation}\label{eqn:cancellation}
\Phi(\omega_1 X_1 X_2 \omega_2)=\Phi(\omega_1 X_1 \omega_2)+\Phi(X_1 X_2)-1.
\end{equation}
\end{lemma}
\begin{proof}
The crucial property of $\Gamma$ that we need for this proof is that $X_1 X_2$ is conjugated to $X_2$ by a group element that can be written as a word $\gamma$ with support contained in $\big(\supp(X_1)\cup\supp(X_2)\big)\setminus\{\tilde{x}\}$.
Indeed, equation~\ref{eqn:pentagram-relation} gives
\begin{align*}
\E{xa}\E{xb}=\E{b\inv{a}}\E{xb}\E{ba},\qquad \E{ax}\E{bx}=\E{ab}\E{bx}\E{a\inv{b}},
\end{align*}
and analogously for $x$ replaced by $\inv{x}$.
Let the notation $X^{(k)}$ denote a star-shaped word around $\tilde{x}$ consisting of $k$-many letters of the same type as $X_1,X_2$, with the understanding that if $X^{(k)}$ is part of a longer word, then no place in $\supp(X^{(k)})\setminus\{\tilde{x}\}$ appears a second time in that word.
The word $X^{(k)}$ is unique up to symmetry.

If $\omega'_1,\omega'_2$ are words with $\supp(\omega'_1)\cap\supp(\gamma)=\supp(\omega'_2)\cap\supp(\gamma)=\emptyset$, then
\begin{equation}\label{eqn:quadrupel-equation}\arraycolsep=0pt
\begin{array}{rlllll}
&\big(\Phi(\omega'_1X_1X_2\omega'_2)&-\Phi(\omega'_1X_1X_2)&\big)-\big(\Phi(\omega'_1\gamma X_1X_2\gamma^{-1}\omega'_2)&-\Phi(\omega'_1\gamma X_1X_2\gamma^{-1})&\big)\\
\overset{\ref{eqn:disjoint-right}}{=}&\big(\Phi(\omega'_1X_1X_2\gamma^{-1}\omega'_2)&-\Phi(\omega'_1X_1X_2\gamma^{-1})&\big)-\big(\Phi(\omega'_1\gamma X_1X_2\gamma^{-1}\omega'_2)&-\Phi(\omega'_1\gamma X_1X_2\gamma^{-1})&\big)\\
=&\big(\Phi(\omega'_1\gamma X_1X_2\gamma^{-1})&-\Phi(\omega'_1X_1X_2\gamma^{-1})&\big)+\big(\Phi(\omega'_1X_1X_2\gamma^{-1}\omega'_2)&-\Phi(\omega'_1\gamma X_1X_2\gamma^{-1}\omega'_2)&\big)\\
\overset{\ref{eqn:disjoint-left}}{=}&\big(\Phi(\gamma X_1X_2\gamma^{-1})&-\Phi(X_1X_2\gamma^{-1})&\big)+\big(\Phi(X_1X_2\gamma^{-1}\omega'_2)&-\Phi(\gamma X_1X_2\gamma^{-1}\omega'_2)&\big)\\
=&\big(\Phi(X_1X_2\gamma^{-1}\omega'_2)&-\Phi(X_1X_2\gamma^{-1})&\big)-(\Phi(\gamma X_1X_2\gamma^{-1}\omega'_2)&-\Phi(\gamma X_1X_2\gamma^{-1})&\big)\\
\overset{\ref{eqn:disjoint-right}}{=}&\big(\Phi(X_1X_2\omega'_2)&-\Phi(X_1X_2)&\big)-\big(\Phi(\gamma X_1X_2\gamma^{-1}\omega'_2)&-\Phi(\gamma X_1X_2\gamma^{-1})&\big)
\end{array}
\end{equation}
follows from equations~\ref{eqn:disjoint-right} and~\ref{eqn:disjoint-left}.
We apply this to the case of $\omega'_1=\omega''_1 X^{(k)}$, $\omega'_2=X^{(1)}$ and find that the value
\begin{equation*}\arraycolsep=0pt
\begin{array}{rlllll}
&\big(\Phi(\omega''_1 X^{(k+3)})&-\Phi(\omega''_1 X^{(k+2)})&\big)-\big(\Phi(\omega''_1 X^{(k+2)})&-\Phi(\omega''_1 X^{(k+1)})&\big)\\
=&\big(\Phi(\omega'_1X_1X_2\omega'_2)&-\Phi(\omega'_1X_1X_2)&\big)-\big(\Phi(\omega'_1\gamma X_1X_2\gamma^{-1}\omega'_2)&-\Phi(\omega'_1\gamma X_1X_2\gamma^{-1})&\big)\\
\overset{\ref{eqn:quadrupel-equation}}=&\big(\Phi(X_1X_2\omega'_2)&-\Phi(X_1 X_2)&\big)-\big(\Phi(\gamma X_1X_2\gamma^{-1}\omega'_2)&-\Phi(\gamma X_1 X_2\gamma^{-1})&\big)\\
=&\big(\Phi(X^{(3)})&-\Phi(X^{(2)})&\big)-\big(\Phi(X^{(2)})&-\Phi(X^{(1)})&\big)
\end{array}
\end{equation*}
does not depend on $k$. Hence, the function $f\colon k\mapsto\Phi(\omega''_1 X^{(k)})$ is of the form $\lambda_2 k^2+\lambda_1 k+\lambda_0$.

Because $\Phi$ has nonnegative values, $\lambda_2\geq 0$. We show that $\lambda_2=0$.
Let $X'_1,\dots,X'_{2n}$ be the letters in $X^{(2n)}$, let $\gamma_1,\dots,\gamma_n$ be such that $\gamma_iX'_{2i-1}X'_{2i}\gamma_i^{-1}=X'_{2i}$, and let $\gamma=\prod\gamma_i$. Because the supports of the $\gamma_i$ are all disjoint, repeated application of equation~\ref{eqn:disjoint-left} shows that $\|\one-\gamma\|^2=\Phi(\gamma)=k$. By the triangle inequality,
\begin{align*}
\sqrt{f(2k)}&=\|\one-\omega''_1 X^{(2k)}\|
=\|\one-\gamma^{-1}\omega''_1 X^{(k)}\gamma\|\\
&\leq\|\one-\gamma^{-1}\|+\|\gamma^{-1}-\gamma^{-1}\omega''_1 X^{(k)}\|+\|\gamma^{-1}\omega''_1 X^{(k)}-\gamma^{-1}\omega''_1 X^{(k)}\gamma\|\\
&=\|\one-\gamma\|+\|\one-\omega''_1 X^{(k)}\|+\|\one-\gamma\|=\sqrt{f(k)}+2\sqrt{k}.
\end{align*}
This is not possible if $\lambda_2>0$ and $k\to\infty$, whence $f$ is linear.

By the triangle inequality, the difference between the two distances $\|\one-X^{(k)}\|$ and $\|(\omega''_1)^{-1}-X^{(k)}\|$ is bounded by $\|\one-\omega''_1\|$, independent of $k$, and since
\begin{equation*}\arraycolsep=0pt
\begin{array}{rll}
\|\one-X^{(k)}\|&=\sqrt{\Phi(X^{(k)})}&\overset{f\text{ lin.}}=\sqrt{\Phi(X^{(1)})+(k-1)\big(\Phi(X^{(2)})-\Phi(X^{(1)})\big)}\qquad\text{and}\\
\|(\omega''_1)^{-1}-X^{(k)}\|&=\sqrt{\Phi(\omega''_1 X^{(k)})}&\overset{f\text{ lin.}}=\sqrt{\Phi(\omega''_1 X^{(1)})+(k-1)\big(\Phi(\omega''_1 X^{(2)})-\Phi(\omega''_1 X^{(1)})\big)},
\end{array}
\end{equation*}
this can only happen when, independent of $\omega''_1$, it holds
\begin{align}\stepcounter{equation}
\Phi(\omega''_1 X^{(2)})-\Phi(\omega''_1 X^{(1)})&=\Phi(X^{(2)})-\Phi(X^{(1)})\qquad\text{and consequently}\label{eqn:cancel-double-right}\tag{\arabic{equation}a}\\ \Phi(X^{(2)}\omega''_1)-\Phi(X^{(1)}\omega''_1)&=\Phi(X^{(2)})-\Phi(X^{(1)})\qquad\text{(using $\Phi(\omega^{-1})=\Phi(\omega)$).}\label{eqn:cancel-double-left}\tag{\arabic{equation}b}
\end{align}
Finally, we obtain
\begin{equation*}\arraycolsep=0pt
\begin{array}{clllll}
&\big(\Phi(\omega_1 X^{(2)}\omega_2)&-\Phi(\omega_1 X^{(1)}\omega_2)&\big)-\big(\Phi(X^{(2)})&-\Phi(X^{(1)})&\big)\\
\overset{\ref{eqn:cancel-double-right}}=&\big(\Phi(\omega_1 X^{(2)}\omega_2)&-\Phi(\omega_1X^{(1)}\omega_2)&\big)-\big(\Phi(\omega_1 X^{(2)})&-\Phi(\omega_1 X^{(1)})&\big)\\
=&\big(\Phi(\omega_1 X^{(2)}\omega_2)&-\Phi(\omega_1 X^{(2)})&\big)-\big(\Phi(\omega_1\gamma X^{(2)}\gamma^{-1}\omega_2)&-\Phi(\omega_1\gamma X^{(2)}\gamma^{-1})&\big)\\
\overset{\ref{eqn:quadrupel-equation}}=&\big(\Phi(X^{(2)}\omega_2)&-\Phi(X^{(2)})&\big)-\big(\Phi(\gamma X^{(2)}\gamma^{-1}\omega_2)&-\Phi(\gamma X^{(2)}\gamma^{-1})&\big)\\
=&\big(\Phi(X^{(2)}\omega_2)&-\Phi(X^{(2)})&\big)-\big(\Phi(X^{(1)}\omega_2)&-\Phi(X^{(1)})&\big)\overset{\ref{eqn:cancel-double-left}}=0,
\end{array}
\end{equation*}
as claimed.
\end{proof}

We proceed according to our intermediate goal of calculating $\Phi(\omega)$ for star-shaped words. The next lemma shows that these words behave to some degree as if they would commute.
Note that the following proof is more specific to the group $\Gamma$ than that of the previous lemma. It uses, for example, that only four different types of letters can appear in a star-shaped word.

\begin{lemma}\label{lem:reorder}
Let $\omega=\omega_1 XY\omega_2$ be a star-shaped word, where $X,Y\in \mathcal{S}$ and $\omega_1,\omega_2$ are words in $\mathcal{S}$ that both contain all four types of letters. Then
\begin{equation}\label{eqn:reorder}
\Phi(\omega_1 XY\omega_2)=\Phi(\omega_1 Y\mkern-3muX\omega_2).
\end{equation}
\end{lemma}
\begin{proof}
If $X$ and $Y$ commute or are of the same type, the statement is obvious. Let $\tilde{x}$ be the recurring place of $\omega$. Then the remaining options for $\{X,Y\}$ are
$\{\E{ax},\E{xb}\}$,
$\{\E{ax},\E{\inv{x}b}\}$,
$\{\E{a\inv{x}},\E{xb}\}$ and
$\{\E{a\inv{x}},\E{\inv{x}b}\}$,
for some $a,b\in\mathcal{Z}\cup\mathcal{Z}^{-1}$.
Since the statement is invariant under $\Lambda$-symmetry and taking inverses, the last three cases can be reduced to the first one, $\{X,Y\}=\{\E{ax},\E{xb}\}$. Note also that the roles of $X$ and $Y$ are interchangeable.
By equation~\ref{eqn:pentagram-relation}, there exists a generator $C=\E{ab}^{\pm 1}$ such that $XY=CY\mkern-3muX$.

Let $Z$ be the first letter in $\omega_2=Z\omega'_2$. We will first assume that $Z$ has the same type as either $X$ or $Y$, say $X$, and let $W$ be another letter of that type. Then
\begin{align*}\arraycolsep=0pt
\begin{array}{clll}
&\Phi(\omega_1 XY\omega_2)&-\Phi(\omega_1 Y\mkern-3muX\omega_2)\\=
&\Phi(\omega_1 CY\mkern-3muX\mkern-2muZ\omega'_2)&-\Phi(\omega_1Y\mkern-3muX\mkern-2muZ\omega'_2)\\
\overset{\ref{eqn:disjoint-left}}=&\Phi(W\mkern-2muXY\mkern-2muZ\omega'_2)&-\Phi(WY\mkern-3muX\mkern-2muZ\omega'_2)\\
\overset{\ref{eqn:cancellation}}=&\big(\Phi(XY\mkern-2muZ\omega'_2)+\Phi(W\mkern-2muX)-1\big)&-\big(\Phi(WY\mkern-3muX\omega'_2)+\Phi(X\mkern-2muZ)-1\big)
=0,
\end{array}
\end{align*}
as claimed.

If, however, $Z$ is of type $\E{\inv{x}c}$ or $\E{c\inv{x}}$, then $Z$ must commute with one of the letters $X,Y$, say $X$, and for the other letter equation~\ref{eqn:pentagram-relation} and its inverse give $Y\mkern-2muZ=ZY\mkern-2muC'$, for a generator $C'$ commuting with $X$.
Indeed, if $Z=\E{\inv{x}c}$, then $Y=\E{ax}$ and $C'=\E{ac}$, while if $Z=\E{c\inv{x}}$, then $Y=\E{xb}$ and $C'=\E{cb}$. Note that in both cases the word $CY\mkern-2muC'$ is star-shaped, and either all three letters are of the same type or two letters are of the same type and the third commutes with both.
Hence, we have
\begin{align*}\arraycolsep=0pt
\begin{array}{clllll}
&\big(\Phi(\omega_1 XY\mkern-2muZ\omega'_2)&-\Phi(\omega_1 Y\mkern-3muX\mkern-2muZ\omega'_2)&\big)-\big(\Phi(\omega_1 XY\omega'_2)&-\Phi(\omega_1 Y\mkern-3muX\omega'_2)&\big)\\
\overset{\ref{eqn:disjoint-left}}=&\big(\Phi(Z^{-1}XY\mkern-2muZ\omega'_2)&-\Phi(Z^{-1}Y\mkern-3muX\mkern-2muZ\omega'_2)&\big)-\big(\Phi(XY\omega'_2)&-\Phi(Y\mkern-3muX\omega'_2)&\big)\\
=&\big(\Phi(XY\mkern-2muC'\omega'_2)&-\Phi(XY\omega'_2)&\big)+\big(\Phi(Y\mkern-3muX\omega'_2)&-\Phi(Y\mkern-3muXC'\omega'_2)&\big)\\
\overset{\ref{eqn:disjoint-right}}=&\big(\Phi(XY\mkern-2muC'X^{-1})&-\Phi(XY\mkern-3muX^{-1})&\big)+(\Phi(Y\mkern-3muXX^{-1})&-\Phi(Y\mkern-3muXC'X^{-1})&\big)\\
=&\big(\Phi(CY\mkern-2muC')&-\Phi(CY)&\big)+\big(\Phi(Y)&-\Phi(Y\mkern-2muC')&\big)\overset{\ref{eqn:cancellation}}=0.
\end{array}
\end{align*}

By the preceding equation, it suffices to deal with the case where the first letter has been removed from $\omega_2$.
We repeat this step of removing letters until we reach a letter in $\omega_2$ that has the same type as $X$ or $Y$. Then the statement follows as before.
\end{proof}

At this point equation~\ref{eqn:adjacent-big} enters the picture and provides another motivation for our interest in star-shaped words. For a fixed $x\in\mathcal{Z}$ and for $z_1,z_2,\ldots\in\mathcal{Z}\setminus\{x\}$ distinct, it follows from $\Lambda$-symmetry that for $k\to\infty$
\begin{equation}\label{eqn:adjacent-exact}
0\leq\langle{\textstyle\frac{1}{k}\sum_{i=1}^k\Delta_{xz_i},\frac{1}{k}\sum_{i=1}^k\Delta_{xz_i}}\rangle\to\langle\Delta_{xz_1},\Delta_{xz_2}\rangle\overset{\ref{eqn:adjacent-big}}\leq 0.
\end{equation}
By the Cauchy--Schwarz inequality this implies that for any word $\omega$ and for $\tilde{a}\notin\supp\omega$
\begin{equation*}\arraycolsep=0pt
\begin{array}{rcl}
0&\overset{\text{CS}}=&\lim\,\langle\one-\omega,\frac{1}{k}\sum_{i=1}^k\Delta_{xz_i}\rangle
=\langle\one-\omega,\Delta_{xa}\rangle\\
&=&\phantom{+}
\langle\one-\omega,\one-\E{xa}\rangle
+\phantom{2}\langle\one-\omega,\one-\E{\inv{x}a}\rangle
+\phantom{2}\langle\one-\omega,\one-\E{ax}\rangle
+\phantom{2}\langle\one-\omega,\one-\E{a\inv{x}}\rangle\\
&&+\langle\one-\omega,\one-\E{x\inv{a}}\rangle
+\phantom{2}\langle\one-\omega,\one-\E{\inv{x}\inv{a}}\rangle
+\phantom{2}\langle\one-\omega,\one-\E{\inv{a}x}\rangle
+\phantom{2}\langle\one-\omega,\one-\E{\inv{a}\inv{x}}\rangle\\
&=&\hspace{\widthof{+}-\widthof{2}}2\langle\one-\omega,\one-\E{xa}\rangle+2\langle\one-\omega,\one-\E{\inv{x}a}\rangle+2\langle\one-\omega,\one-\E{ax}\rangle+2\langle\one-\omega,\one-\E{a\inv{x}}\rangle\\
&=&\phantom{+}\big(\|\one-\omega\|^2+\|\one-\E{xa}\|^2-\|\E{xa}-\omega\|^2\big)+\big(\|\one-\omega\|^2+\|\one-\E{\inv{x}a}\|^2-\|\E{\inv{x}a}-\omega\|^2\big)\\
&&+\big(\|\one-\omega\|^2+\|\one-\E{ax}\|^2-\|\E{ax}-\omega\|^2\big)+\big(\|\one-\omega\|^2+\|\one-\E{a\inv{x}}\|^2-\|\E{a\inv{x}}-\omega\|^2\big)\\
&=&4\cdot\Phi(\omega)+4-\Phi(\E{x\inv a}\omega)-\Phi(\E{\inv{x}\inv a}\omega)-\Phi(\E{a\inv{x}}\omega)-\Phi(\E{ax}\omega).
\end{array}
\end{equation*}
Applying the above equation to $\omega$ and $\omega^{-1}$ gives
\begin{align}\stepcounter{equation}
\label{eqn:expand-left}
4\cdot\Phi(\omega)+4=\Phi(\E{xa}\omega)+\Phi(\E{\inv{x}a}\omega)+\Phi(\E{ax}\omega)+\Phi(\E{a\inv{x}}\omega)\tag{\arabic{equation}a},\\
\label{eqn:expand-right}
4\cdot\Phi(\omega)+4=\Phi(\omega \E{xa})+\Phi(\omega \E{\inv{x}a})+\Phi(\omega \E{ax})+\Phi(\omega \E{a\inv{x}}).\tag{\arabic{equation}b}
\end{align}

Note that when $\omega$ is star-shaped around $\tilde{x}$, the same is true for all words on the right hand sides of equations~\ref{eqn:expand-left} and~\ref{eqn:expand-right}.

\begin{remark}
Equation~\ref{eqn:adjacent-exact} can be viewed as a special case of a more general tactic to produce information about $\Phi$.
Namely, define the Caesaro means
\begin{equation*}
M^k_{x\cdot}\vcentcolon=\frac{1}{k}\sum_{i=1}^k \E{xz_i},\quad
M^k_{\inv{x}\cdot}\vcentcolon=\frac{1}{k}\sum_{i=1}^k \E{\inv{x}z_i},\quad
M^k_{\cdot x}\vcentcolon=\frac{1}{k}\sum_{i=1}^k \E{z_i x},\quad
M^k_{\cdot\inv{x}}\vcentcolon=\frac{1}{k}\sum_{i=1}^k \E{z_i\inv{x}}.
\end{equation*}
When $\varphi_k,\psi_k\in\R\Gamma$ are products in $\{M^k_{x\cdot},M^k_{\inv{x}\cdot},M^k_{\cdot x},M^k_{\cdot\inv{x}}\}\cup\mathcal{S}$, then the inner product $\langle\varphi_k-\one,\psi_k-\one\rangle$ expands into a sum that can be simplified with $\Lambda$-symmetry and that converges to the value of its dominating summand as $k\to\infty$. One can think of $\lim\varphi_k-\one$ and $\lim\psi_k-\one$ as weak limits in the Hilbert space induced by $\langle\cdot,\cdot\rangle$.
These limit points can be used in the same way as regular group elements in both human arguments and the SDP computer proofs, in order to obtain additional inequalities.
\end{remark}

\begin{lemma}\label{lem:star-shaped-values}
Let $\omega$ be a star-shaped word of length $|\omega|$. Then $\Phi(\omega)=|\omega|$.
\end{lemma}
\begin{proof}
Let $\omega=\omega'X$ be any star-shaped word around $\tilde{x}$, where $X$ is a single letter. By applying equations~\ref{eqn:expand-left} and~\ref{eqn:expand-right} iteratively, we obtain
\begin{align*}
\Phi(\omega')&+2k=\frac{1}{4^k\cdot 4^k}\cdot\sum\Phi\big(
\omega_1\omega'\omega_2\big)\quad\text{and}\\
\Phi(\omega'X)&+2k=\frac{1}{4^k\cdot 4^k}\cdot\sum\Phi\big(
\omega_1\omega'X\omega_2\big),
\end{align*}
where the sums range over all $k$-long words $\omega_1,\omega_2$ over the letters $\E{xa},\E{\inv{x}a},\E{ax},\E{a\inv{x}}$, and $\tilde{a}$ becomes a new unused place for every letter.
Every word appearing in the sums is star-shaped around $\tilde{x}$, and most summands are good in the sense that both $\omega_1$ and $\omega_2$ contain all four types of letters at least twice.

For every good summand we can write $\omega_1=\omega'_1X'\omega''_1$, where $X'$ is a single letter of the same type as $X$ and where $\omega'_1$ contains all four types of letters. Then
\begin{equation*}
\Phi(\omega_1\omega'X\omega_2)-\Phi(\omega_1\omega'\omega_2)
\overset{\ref{eqn:reorder}}=
\Phi(\omega'_1\omega''_1X'X\omega'\omega_2)-\Phi(\omega'_1\omega''_1X'\omega'\omega_2)
\overset{\ref{eqn:cancellation}}=\Phi(X'X)-\Phi(X).
\end{equation*}
The proportion of bad summands decreases exponentially as $k\to\infty$, and by equation~\ref{eqn:quadratic-bound} the error that they contribute increases at most quadratically. Hence, letting $k\to\infty$ gives
$\Phi(\omega'X)-\Phi(\omega')=\Phi(X'X)-\Phi(X)$.
Applying this to $\omega'=\one$ gives $\Phi(X'X)-\Phi(X)=1$, which, in turn, simplifies the last equation to $\Phi(\omega'X)=\Phi(\omega')+1$. The statement follows by induction on $|\omega|$.
\end{proof}

To conclude the proof of Theorem~\ref{thm:main-result}, consider the diagram in Figure~\ref{fig:diagram}, which is the analogue of the left diagram in \cite{N}*{Figure 1}. Two group elements $\gamma_1,\gamma_2$ are connected by a solid line if $\|\gamma_1-\gamma_2\|^2=1$, and by a dashed line if Lemma~\ref{lem:star-shaped-values} shows $\|\gamma_1-\gamma_2\|^2=2$. Here we use that $\E{ac}$ commutes with $\E{bc}$, the only place in this proof where we need equation~\ref{eqn:right-commutes}.
The elements $\one,\E{ac},\E{bc},\E{ac}\E{bc}$ form a planar square, which is perpendicular to the line through $\one,\E{ab}$. But then the value of $\|\E{ac}\E{bc}-\E{ab}\|^2=\Phi(\E{a\inv{c}}\E{b\inv{c}}\E{ab})=\Phi(\E{ab}\E{b\inv{c}})$ cannot be $2$, as claimed. This contradiction means that the initial assumption was false and all but finitely many of the groups $\Gamma_n$ have property~(T).

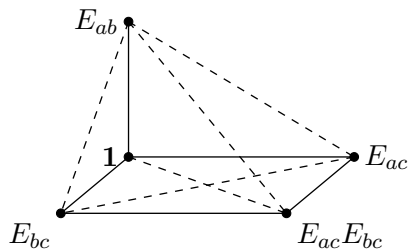
\begin{figure}[ht]\label{fig:diagram}
\begin{tikzpicture}[x=0.30cm, y=0.30cm]
\coordinate (P1) at (3.0, 2.50);
\coordinate (Pac) at (13.0, 2.50);
\coordinate (Pbc) at (0.0, 0.0);
\coordinate (Pacbc) at (10.0, 0.0);
\coordinate (Pab) at (3.0, 8.50);
\draw[line width=0.5pt] (P1) -- (Pac);
\draw[line width=0.5pt] (P1) -- (Pbc);
\draw[line width=0.5pt] (Pac) -- (Pacbc);
\draw[line width=0.5pt] (Pbc) -- (Pacbc);
\draw[line width=0.5pt] (P1) -- (Pab);
\draw[dashed, line width=0.5pt] (P1) -- (Pacbc);
\draw[dashed, line width=0.5pt] (Pac) -- (Pbc);
\draw[dashed, line width=0.5pt] (Pab) -- (Pac);
\draw[dashed, line width=0.5pt] (Pab) -- (Pbc);
\draw[dashed, line width=0.5pt] (Pab) -- (Pacbc);
\draw[fill=black] (P1)    circle (0.15em) node[left] {$\one$};
\draw[fill=black] (Pac)   circle (0.15em) node[right] {$\E{ac}$};
\draw[fill=black] (Pbc)   circle (0.15em) node[below left] {$\E{bc}$};
\draw[fill=black] (Pacbc) circle (0.15em) node[below right] {$\E{ac}\E{bc}$};
\draw[fill=black] (Pab)   circle (0.15em) node[left] {$\E{ab}$};
\end{tikzpicture}
\caption{The distances between the elements $\one,\E{ac},\E{bc},\E{ac}\E{bc},\E{ab}$}
\end{figure}

\begin{remark}
If we put a total order on $\mathcal{Z}$ such that infinitely many places appear both before and after $\tilde{x}$, then all proof steps in this section require only words with letters $\E{ab}$ where $\tilde{a}<\tilde{b}$.
This is analogous to the case of $\Gamma=\SL(3,\Z)$ (see \cite{N}), where the property~(T) proof can be carried out by looking only at upper triangular matrices in combination with the symmetry group $\Lambda$ of $(\Gamma,\mathcal{S})$.
\end{remark}

\begin{question}
Can the above property~(T) proof be carried out without using relation~\ref{eqn:right-commutes}, i.e., is it true that for all but finitely many $n$ the group with generators $\mathcal{S}_n$ and relations~\ref{eqn:inverse} to~\ref{eqn:pentagram-relation} has property~(T)?
\end{question}

\section*{Acknowledgments}

This research was supported by DFG Grant 281869850 (RTG 2229, ``Asymptotic Invariants and Limits of Groups and Spaces'').

\end{document}